\documentclass[preprint,3p]{elsarticle}
\biboptions{sort&compress}

\usepackage[pdftex,colorlinks]{hyperref}
 \usepackage[pagewise,mathlines]{lineno}

\journal{arxiv}

\usepackage{amsfonts}
\usepackage{amssymb}
\usepackage[utf8]{inputenc}
\usepackage{amsmath}
\usepackage{amsthm}
\usepackage{graphicx}%
\usepackage{enumerate}
\usepackage{xcolor}
\usepackage{color,soul}
\usepackage{tikz}
\usetikzlibrary{patterns}
\usepackage{mathtools}

\theoremstyle{definition}
\newtheorem{teorema}{Theorem}[section]
\newtheorem{lema}[teorema]{Lemma}
\newtheorem{corolario}[teorema]{Corollary}
\newtheorem{proposicao}[teorema]{Proposition}
\newtheorem{definicao}[teorema]{Definition}

\theoremstyle{remark}
\newtheorem{observacao}[teorema]{Remark}

\numberwithin{equation}{section}

\newcommand{\R}{\mathbb{R}}

\newcommand{\rhoi}[1]{\rho_{_{#1}}}
\newcommand{\n}[2]{\|#1\|_{_{#2}}}
\newcommand{\intw}{\int_0^T\int_{\omega}}	
	
\newcommand{\into}{\int_0^1}

\newcommand{\nn}[1]{\|#1\|}

\newcommand{\intq}{\int_0^T\int_0^1}

\newcommand{\lnum}{\let\veqno\leqno}
	
\newcommand{\dom}{Q}
\newcommand{\domw}{\omega_T}

\newcommand{\rz}{\rhoi{1}}
\newcommand{\G}[1]{\int_0^T\int_0^1e^{-2sA}\left(s\lambda^2\zeta x^\alpha |#1_x|^2+s^3\lambda^{4}\zeta^{2} |#1|^2 \right) \dd x \dd t  }

\newcommand{\dd}{\mathop{}\!d}

\newcommand{\bu}{\bar{u}}
\newcommand{\bh}{\bar{h}}

\usepackage[textsize=scriptsize,textwidth=15mm]{todonotes}
\usepackage{marginnote}
\makeatletter
\renewcommand{\@todonotes@drawMarginNoteWithLine}{%
	\begin{tikzpicture}[remember picture, overlay, baseline=-0.75ex]%
	\node [coordinate] (inText) {};%
	\end{tikzpicture}%
	\marginnote[{
		\@todonotes@drawMarginNote%
		\@todonotes@drawLineToLeftMargin%
	}]{
		\@todonotes@drawMarginNote%
		\@todonotes@drawLineToRightMargin%
	}%
}
\makeatother

\allowdisplaybreaks[1] 

\begin{document}

\begin{frontmatter}


\title{\textsc{Carleman inequality for a class of super strong degenerate  parabolic operators and applications.}}

\author[UFCG]{B. S. V. Ara\'ujo}
\ead{bsergio@mat.ufcg.edu.br}

\author[RCN]{R. Demarque\corref{mycorrespondingauthor}}
\cortext[mycorrespondingauthor]{Corresponding author}
\ead{reginaldo@id.uff.br}

\author[GAN]{L. Viana}
\ead{luizviana@id.uff.br}

\address[UFCG]{Unidade Acadêmica de Matemática, Universidade Federal de Campina Grande, Campina Grande, PB, Brazil}
\address[RCN]{Departamento de Ciências da Natureza,
	Universidade Federal Fluminense,
	Rio das Ostras, RJ, Brazil}
\address[GAN]{Departamento de Análise,
	Universidade Federal Fluminense,
	Niter\'{o}i, RJ, Brazil}

\begin{abstract}
In this paper, we present a new Carleman estimate for the adjoint equations associated to a class of super strong degenerate parabolic linear problems. Our approach considers a standard geometric imposition on the control domain, which can not be removed in the general situations. Additionally, we also apply the aformentioned main inequality in order to investigate the null controllability of two nonlinear parabolic systems. The first application is concerned a global null controllability result obtained for some semilinear equations, relying on a fixed point argument. In the second one, a local null controllability for some equations with nonlocal terms is also achieved, by using an inverse function theorem.   
\end{abstract}


\end{frontmatter}


\section{Introduction}\label{introd}

   In this work we derive a new Carleman estimate for the linear super strong degenerate problem
   \begin{equation}\label{prob0}
       \begin{cases}
          u_t-(x^\alpha u_x)_x+x^{\alpha/2}b_1(x,t)u_x+b_0(x,t)u  =  f1_{\omega} &\text{in }    Q,\\ u(1,t)=0  \text{ and }  (x^\alpha u_x)(0,t)=0  & \text{in }  (0,T),\\ u(x,0)= u_0(x)  & \text{in } (0,1),
       \end{cases}
   \end{equation}
      where $Q=(0,1)\times(0,T)$, $\omega\subset(0,1)$ is a non-empty open interval and $1_\omega$ is its associated characteristic function and $\alpha\geq2$. Also,  $b_0\in L^\infty(Q)$, $h\in L^2(\omega\times (0,T))$,  $u_0\in L^2(0,1)$, $b_1\in L^\infty(Q)$ and satisfy:
   \begin{equation}\label{hipfunc1}
       (x^{\alpha/2}b_1(x,t))_x\in L^\infty(Q).
   \end{equation}
And, we also consider a geometrical condition on the control domain
   \begin{equation}
\label{geohip}\exists d>0; \ (0,d)\subset\omega.
\end{equation}
 
 As we will see further, the problem \eqref{prob0} is controllable at any time $T>0$, according to the following specification: 
\begin{definicao} We say that \eqref{prob0}, is {\it null controllable} if, for any $u_0\in L^2(0,1)$, there exists $h\in L^2(\omega\times (0,T))$ such that the solution $u$ of \eqref{prob0} satisfy
     \begin{equation}\label{null}u(x,T)=0\ \ \ \mbox{in} \ \ (0,1).\end{equation}\end{definicao}

The null controllability of problem \eqref{prob0} for $\alpha\in (0,2)$ is well understood, see \cite{alabau2006carleman,cannarsa2016global} and references therein. Following the terminology adopted in these works, we say that \eqref{prob0} is {\it weakly degenerate} if $\alpha\in(0,1)$ and is {\it strongly degenerate} if $\alpha\in(1,2)$. Despite there are many works for the case $\alpha\in(0,2)$, little has been done for the {\it super strong degenerate case}, i.e. when $\alpha\geq 2$, although this is a very relevant case of the degenerate problem. Indeed, when $\alpha=2$, the Black-Scholes equation can be obtained from \eqref{prob0} and this equation has a key role in several financial problems.

Regarding the global null controllability of \eqref{prob0}, the fact is that this problem is not, in general, null controllable for $\alpha\geq2$. As pointed out in \cite{cannarsa2016global}, a suitable change of variables transform the problem \eqref{prob0} into a non-degenerate problem in an unbounded domain, which fails to be null controllable in general, as proved in \cite{micu2001lack}. However, if the new control domain $\tilde{\omega}$ has bounded complement, it can be controlled, as proved in  \cite{cabanillas2001null,cannarsa2004lebesgue}.


Because of that, in \cite{cannarsa2004persistent} was introduced  a weaker kind of null controllability, {\it regional null controllability} for this problem. It means that for any $u_0\in L^2(0,T)$, $\omega=(a,b)\subset(0,1)$ and $\delta\in (0,b-a)$, there exists a control $f\in L^2(Q)$ such that the solution $u$ of \eqref{prob0} satisfies 
\begin{equation}
\label{nullreg}u(x,T)=0 \ \ \forall x\in (a+\delta,1). 
\end{equation}
However, in order to establish a global null controllability result for \eqref{prob0} with $\alpha=2$, the authors in  \cite{araruna2018stackelberg} came up with the new geometrical condition \eqref{geohip}. In this work, considering the same geometrical condition, we were able to extend that result for $\alpha >2$.

A  significant number of papers on null controllability of parabolic degenerate equations follows a standard approach based on the Hilbert Uniqueness Method (HUM). It goes through obtaining a Carleman estimate that leads to an observabillity inequality. This way, the  null controllability property can be deduced from the observabillity inequality. The particularity of \cite{araruna2018stackelberg} and \cite{cannarsa2004persistent} is that the authors applied a change of variables  to transform the system \eqref{prob0} into a non-degenerate problem in unbounded domains. There, a Carleman estimate is obtained for this non-degenerate system. 

Although the approach of transforming the degenerate problem into a non-degenerate one, in an unbounded domain, works fine for linear problems, this procedure can meet  difficulties to deal with some related problems. Indeed, when we work with some autonomous semilinear problems, for example, this change of variable leads it to a nonautonomous semilinear problem. And, if we work with a certain nonlocal problems, it is lead to an even more complicated one. In this work we present a Carleman estimate for \eqref{prob0}, without passing by this change of variables method.  To our best knowledge, this estimate and some consequences presented in the sequel mean some novelties for the super strong degenerate case.
 


The second part of the introduction is all about the presentation of our main resuts.

\subsection*{Statement of the Results}


First of all, let us consider the adjoint system associated to \eqref{prob0}:
\begin{equation} \label{prob0'}
\begin{cases}
v_t+(x^\alpha v_x)_x+(x^{\alpha/2}b_1v)_x-b_0(x,t)v  =  h  & \text{in }  Q,\\
v(1,t)=0 \ \  \text{and} \ \  (x^\alpha v_x)(0,t)=0  &  \text{in }   (0,T),\\ 
v(x,T)= v_T(x)   & \text{in }   (0,1),  
\end{cases}
\end{equation}
where $h\in L^2(Q)$ and $v_T\in L^2(0,1)$.

Now, for $\lambda>0$, let us introduce some weight functions given by $\theta$, $p_0$ and $\sigma_0$ with
\begin{multline}\label{functions}
    \theta(t):=\frac{1}{(t(T-t))^4},\  \eta(x):=-x^2/2, \   \xi(x,t)=\theta(t)e^{\lambda(2|\eta|_\infty+\eta(x))} \\  \text{and} \ \sigma(x,t):=\theta(t)e^{4\lambda|\eta|_\infty}-\xi(x,t).
\end{multline}

The assumption \eqref{geohip} and the weight function $\eta$ are the key points that allow us to build the following Carleman estimate:

\begin{teorema}\label{theo1}
Assume \eqref{hipfunc1} and \eqref{geohip}. There exists positive constants $C$, $s_0$ and $\lambda_0$, depending only on $\omega$, $\|b_0\|_\infty$, $T$, $d$ and $\alpha$ such that, for any $s\geq s_0$, any $\lambda\geq\lambda_0$ and any solution $v$ to (\ref{prob0'}), one has:
		\begin{multline}\label{carleman1}
		\iint_Qe^{-2s\sigma}\left[s^{-1}\lambda^{-1}\xi^{-1}(|v_t|^2+|(x^\alpha v_x)_x|^2)+s\lambda^2\xi x^\alpha |v_x|^2+s^3\lambda^4\xi^3|v|^2\right]\,dx\,dt  \\ 
		\leq  C\left[\|e^{-s\sigma}h\|^2+s^3\lambda^4\iint_{\omega _T}e^{-2s\sigma}\xi^3|v|^2\,dx\,dt\right], 
	\end{multline}
	where $\omega _T := \omega \times (0,T)$.
\end{teorema}

The proof of Theorem \ref{theo1} will be given in section \ref{CarObs}. 



As a consequence of Theorem \ref{theo1} we have the following null controllability result: 

\begin{teorema}	\label{main}
Assume \eqref{hipfunc1} and \eqref{geohip}. Then the system \eqref{prob0} is null controllable.
\end{teorema}

The same Carleman estimate allow us to prove a null controllability result for the following semilinear problem
\begin{equation}\label{prob1}
       \begin{cases}
          u_t-(x^\alpha u_x)_x+g(x,t,u,u_x)  =  f1_{\omega} &\text{in }    Q,\\ u(1,t)=0  \text{ and }  (x^\alpha u_x)(0,t)=0  & \text{in }  (0,T),\\ u(x,0)= u_0(x)  & \text{in } (0,1),
       \end{cases}
\end{equation}
where $\alpha\geq2$ and $g:Q\times\mathbb{R}^2\to\mathbb{R}$ must satisfies the following assumptions:
   \begin{equation}\label{hipfunc2}
       \begin{cases}
           g \ \ \mbox{is Lebesgue measurable};\\
           g(x,t,\cdot ,\cdot )\in C^1(\mathbb{R}^2) \ \mbox{uniformly in } (x,t)\in Q;\\
           g(x,t,0,0)=0 \ \ \forall (x,t)\in Q;\\
           \exists K>0 \ \ \mbox{such that} \ \ |g_r(x,t,r,s)|+x^{-\alpha/2}|g_s(x,t,r,s)|\leq K \ \ \forall (x,t,r,s)\in Q\times\mathbb{R}^2.
       \end{cases}
   \end{equation}

In \cite{xu2020null}, a null controllability result is obtained for \ref{main1}, when $\alpha \in (0,2)$. In this current work, we extend this fact for the super strong degenerate case applying a similar technique. To do so, we assume \eqref{hipfunc1} and \eqref{geohip} in order to prove the crucial Carleman estimate \eqref{carleman1}. Precisely, we have:

\begin{teorema}
	\label{main1}Assume \eqref{geohip} and \eqref{hipfunc2}. Then the system \eqref{prob1} is null controllable.
\end{teorema}

As a second application of our Carleman estimate \eqref{carleman1}, we will also obtain the locall null controllability for the following degenerate nonlocal problem
 \begin{equation}\label{pbNL}
       \begin{cases}
          u_t-\ell\left(\int_0^1u\right)\left(x^\alpha u_x\right)_x =  f1_{\omega} &\text{in }    Q,\\ u(1,t)=0  \text{ and }  (x^\alpha u_x)(0,t)=0  & \text{in }  (0,T),\\ u(x,0)= u_0(x)  & \text{in } (0,1),
       \end{cases}
   \end{equation}
   where  $\ell:\R \longrightarrow \R$ is a $C^1$ function with bounded derivative, with $\ell(0)=1$. At this point, we should observe that our results remain the same if we just consider $\ell (0)>0$. The null controllability for this problem is studied in \cite{jrl2016}, when $\alpha\in(0,1)$,  and in \cite{CDJL2021} when $\alpha \in [1,2)$. Under the hypotheses \eqref{hipfunc1} and \eqref{geohip}, we extend this invertigation for $\alpha\in [2,+\infty)$, as follows:

	

\begin{teorema}\label{th-nonlocal}
Assume \eqref{geohip}. The nonlinear system \eqref{pbNL} is \textit{locally null-controllable} at any time $T>0$, i.e, there exists $\varepsilon>0$ such that, whenever $u_0\in H_\alpha^1$ and $|u_0|_{H_\alpha^1}\leq \varepsilon$, there exists a control $f\in L^2(\omega\times (0,T))$, associated to a state $u$, satisfying
		\begin{equation*}
		u(x,T)=0, \text{ for every } x\in [0,1].
		\end{equation*}  
\end{teorema}

The rest of this paper is organized as follows. In Section \ref{pre}, we state some classical well-posedness results related to the systems \eqref{prob0} and \eqref{prob1}. In Section \ref{CarObs}, we present an $\alpha$-independent Carleman inequality for solutions of \eqref{prob0'} (see Theorem \ref{theo1}), which provides an observability estimate and, consequently, the null controllability of \eqref{prob0}. Sections \ref{provasemilinear} and \ref{provanaolocal} are devoted to some applications of  Theorem \ref{theo1}. More precisely, in Section \ref{provasemilinear}, we use a fixed point argument to obtain a null controllability result to the degenerate semilinear system \eqref{prob1} (see Theorem \ref{main1}); in Section \eqref{provanaolocal}, an inverse function argument allows us to prove a local null controllability result for the degenerate nonlocal system \eqref{pbNL} (see Theorem \ref{th-nonlocal}).


    
    
    
    


\section{Well-posedeness results}\label{pre}


   The usual norm and inner product in $L^2(0,1)$ and $L^2(Q)$ will be denoted respectively by $|\cdot|_2$ and $(\cdot,\cdot)$, and  by $\|\cdot\|_2$ and $(\!(\cdot,\cdot)\!)$. Moreover, the norms in $L^\infty(0,1)$ and in $L^\infty(Q)$ will be denoted respectively by $|\cdot|_\infty$ and $\|\cdot\|_\infty$.
   
   Let us consider the functional sets
   \begin{equation*}H^1_\alpha := \{\ u\in L^2(0,1); u \text{ is locally  absolutely continuous in} \ (0,1],\ x^{\alpha/2}u_x\in L^2(0,1), \ u(1)=0\ \}.\end{equation*} and 
   \begin{equation*} 
   H^2_\alpha   := \{\ u\in H^1_\alpha; \ x^\alpha u_x\in H^1(0,1)\}\\
     \end{equation*}
   with the norms \begin{equation*}|u|_{H_\alpha^1}:= \left[\int_0^1(u^2+x^\alpha|u|^2)\,dx\right]^{1/2}, \  u\in H_\alpha^1\end{equation*} and \begin{equation*}|u|_{H_\alpha^2}:= \left[\int_0^1(u^2+x^\alpha|u|^2+|(x^\alpha u_x)_x|^2)\,dx\right]^{1/2}, \  u\in H_\alpha^2.\end{equation*}
   
   With these norms, we have that $H^1_\alpha$ and $H_\alpha^2$ are Hilbert spaces.
      In \cite[Proposition 2.1]{cannarsa2004persistent} the authors provided the following characterization: 
   \begin{multline*}
   H^2_\alpha= \{u\in L^2(0,1);\ u \text{ is locally  absolutely continuous in} \ (0,1],\\ 
   \ x^\alpha u\in H_0^1(0,1),\ x^\alpha u_x\in H^1(0,1) \ (x^\alpha u_x)(0)=0 \} 
   \end{multline*}
   
   Now, for reader's convenience, let us introduce the notations $$\mathcal{M}=C(0,T;L^2(0,1))\cap L^2(0,T;H_\alpha^1) \ \ \ \ \mbox{and} \ \ \ \ \mathcal{N}=H^1(0,T;L^2(0,T))\cap L^2(0,T;H_\alpha^2).$$
   In \cite{xu2020null}, the authors proved that the embedding $\mathcal{M}\hookrightarrow\mathcal{N}$ is compact (actually, their result was proved for $\alpha\in(0,2)$, but the proof does not depend on $\alpha$).
   

   
   

The next result establishes the well-posedness of system \eqref{prob0} and it has been present in \cite{cannarsa2004persistent}.

\begin{proposicao}\label{prop1}
	Assume $b_0,b_1\in L^\infty(Q)$. For any $f\in L^2(Q)$ and any $u_0\in L^2(0,1)$, there exists exactly one solution $u\in \mathcal{M}$ to the system (\ref{prob0}). Furthermore, there exists a constant $C>0$ only depending on $T$, $\alpha$, $b_1$ and $b_0$, such that
	\[\sup_{t\in[0,T]}|u(\cdot,t)|_2^2+\|x^{\alpha/2}u_x\|_2^2\leq C(\|f1_\omega\|^2_2+|u_0|_2^2).\]
	Furthermore, if $u_0\in H_\alpha^1$, then $u\in \mathcal{N}\cap C^0([0,T];H^1_\alpha)$ and we have the following estimate: \[\sup_{t\in[0,T]}|u(\cdot,t)|_{H^1_\alpha}^2+\|u_t\|_2^2+\|(x^\alpha u_x)_x\|_2^2\leq C\left(\|f1_\omega\|_2^2+|u_0|_{H^1_\alpha}^2\right).\]
\end{proposicao}

We also state the well-posedness of system \eqref{prob1}, which prove can be seen in \cite[Theorem 2.1]{xu2020null}.
\begin{proposicao}
 Assume $g$ satisfies \eqref{hipfunc2}. For any $f\in L^2(Q)$ and any $u_0\in L^2(0,1)$, there exists exactly one solution $u\in \mathcal{M}$ to the system (\ref{prob1}).
    
\end{proposicao}

\section{Carleman and Observability Inequalities}\label{CarObs}



 




The aim of this section is to prove the Carleman estimate \eqref{carleman1} and, as a consequence, an observability inequality, which yields the null controllability of the linear system \eqref{prob0}.

It suffices to prove Theorem \ref{theo1} for  $b_1=b_0=0$, since the general case follows from that by taking $\tilde{h}=h-b_0v-(x^\alpha b_1 v_x)_x$. 

Let us take $\delta\in(0,d)$ and let $v$ be the solution to (\ref{prob0'}) (where $v_T\in L^2(0,1)$ and $h\in L^2(Q)$). For any $s\geq s_0>0$, we set $z=e^{-s\sigma}v$. By a density argument we can assume without loss of generality that $v$ is  regular enough. A simple computation gives us  \[v_t=e^{s\sigma}[s\sigma_tz+z_t] \ \ \mbox{and} \ \  (x^\alpha v_x)_x=e^{s\sigma}[s^2\sigma_x^2x^\alpha z+2s\sigma_xx^\alpha z_x+s(\sigma_xx^\alpha)_xz+(x^\alpha z_x)_x].\] Consequently,
    \begin{equation}\label{eq4.1.1}P^+z+P^-z=g,\end{equation}where
    \[P^-z:=-2s\lambda^2\xi x^{\alpha+2}z+2s\lambda\xi x^{\alpha+1}z_x+z_t:=I_{11}+I_{12}+I_{13},\]
    \[P^+z:=s^2\lambda^2\xi^2x^{\alpha+2}z+(x^\alpha z_x)_x+s\sigma_tz:=I_{21}+I_{22}+I_{23}\] and
    \[g=e^{-s\sigma}h-s\lambda^2\xi x^{\alpha+2}z-(\alpha+1)s\lambda\xi x^\alpha z.\]
    
    From \eqref{eq4.1.1} one has \begin{equation}
    \label{eq4.2.1}\|P^-z\|_2^2+\|P^+z\|_2^2+2(\!(P^-z,P^+z)\!)=\|g\|_2^2.\end{equation}
    
    Now let us estimate $(\!(P^-z,P^+z)\!)$. We have that
\[(\!(I_{11},I_{21})\!)=-2s^3\lambda^4\iint_Q\xi^3x^{2\alpha+4}|z|^2\,dx\,dt,\]
\begin{align*}
    (\!(I_{12},I_{21})\!)&=s^3\lambda^3\iint_Q\xi^3x^{2\alpha+3}(|z|^2)_x\,dx\,dt\\ &=3s^3\lambda^4\iint_Q\xi^3x^{2\alpha+4}|z|^2\,dx\,dt-(2\alpha+3)s^3\lambda^3\iint_Q\xi^3x^{2\alpha+2}|z|^2\,dx\,dt
\end{align*}
and
\[(\!(I_{13},I_{21})\!)=\frac{1}{2}s^2\lambda^2\iint_Q\xi^2x^{\alpha+2}(|z|^2)_t\,dx\,dt=-s^2\lambda^2\iint_Q\xi\xi_tx^{\alpha+2}|z|^2\,dx\,dt.\]
Thus
\begin{multline*}
(\!(P^-z,I_{21})\!)= s^3\lambda^4\iint_Q\xi^3x^{2\alpha+4}|z|^2\,dx\,dt-(2\alpha+3)s^3\lambda^3 \iint_Q\xi^3x^{2\alpha+2}|z|^2\,dx\,dt\\  -s^2\lambda^2\iint_Q\xi\xi_tx^{\alpha+2}|z|^2\,dx\,dt.
\end{multline*} 
Since $|\xi\xi_t|\leq C\xi^3$, for $\lambda_0$ and $s_0$ large
 enough, we can deduce that 
\begin{align}\label{eq4.3}
(\!(P^-z,I_{21})\!)&
\begin{multlined}[t]
\geq s^3\lambda^4\int_0^T\left[\int_0^\delta\xi^3x^{2\alpha+4}|z|^2\,dx +\int_\delta^1\xi^3x^{2\alpha+4}|z|^2\,dx\right]\,dt\nonumber\\ 
 -Cs^3\lambda^3\left((2\alpha+3)+\frac{C}{\lambda_0s_0}\right)\iint_Q\xi^3|z|^2\,dx\,dt\nonumber
\end{multlined}\\
& \geq s^3\lambda^4\int_0^T\int_\delta^1\xi^3x^{2\alpha+4}|z|^2\,dx\,dt
 -Cs^3\lambda^3\iint_Q\xi^3|z|^2\,dx\,dt\nonumber\\
 & \geq \delta^{2\alpha+4}s^3\lambda^4\int_0^T\int_\delta^1\xi^3|z|^2\,dx\,dt
 -Cs^3\lambda^3\iint_Q\xi^3|z|^2\,dx\,dt\nonumber\\
 &\geq Cs^3\lambda^4\int_0^T\int_\delta^1\xi^3|z|^2\,dx\,dt-Cs^3\lambda^3\iint_Q\xi^3|z|^2\,dx\,dt\nonumber\\ 
 &\begin{multlined}[t]
  =Cs^3\lambda^4\iint_Q\xi^3|z|^2\,dx\,dt-Cs^3\lambda^4\int_0^T\int_0^\delta\xi^3|z|^2\,dx\,dt\\ 
  -Cs^3\lambda^3\iint_Q\xi^3|z|^2\,dx\,dt
 \end{multlined} 
\nonumber\\ 
 &\geq Cs^3\lambda^4\left(1-\frac{1}{\lambda_0}\right)\iint_Q\xi^3|z|^2\,dx\,dt-Cs^3\lambda^4\int_0^T\int_0^\delta\xi^3|z|^2\,dx\,dt\nonumber\\
 &\geq Cs^3\lambda^4\iint_Q\xi^3|z|^2\,dx\,dt-Cs^3\lambda^4\int_0^T\int_0^\delta\xi^3|z|^2\,dx\,dt.
\end{align}
Note that $C$ depend on $\delta$ and $\alpha$, where $\delta$ is a fixed number in $(0,d)$. 

Furthermore,
\[(\!(I_{11},I_{23})\!)=-2s^2\lambda^2\iint_Q\xi\sigma_tx^{\alpha+2}|z|^2\,dx\,dt,\]
\begin{align*}(\!(I_{12},I_{23})\!)
&=s^2\lambda\iint_Q\xi\sigma_tx^{\alpha+1}(|z|^2)_x\,dx\,dt\\ &=s^2\lambda^2\iint_Q\xi(\sigma_t+\xi_t)x^{\alpha+2}|z|^2\,dx\,dt-(\alpha+1)s^2\lambda\iint_Q\xi\sigma_tx^\alpha|z|^2\,dx\,dt
\end{align*}
and
\[(\!(I_{13},I_{23})\!)=\frac{s}{2}\iint_Q\sigma_t(|z|^2)_t\,dx\,dt=-\frac{s}{2}\iint_Q\sigma_{tt}|z|^2\,dx\,dt.\]
Thus
\begin{equation*}
(\!(P^-z,I_{23})\!)=-s^2\lambda^2\iint_Q\xi(\xi_t+\sigma_t)x^{\alpha+2}|z|^2\,dx\,dt-(\alpha+1)s^2\lambda\iint_Q\xi\sigma_tx^\alpha|z|^2\,dx\,dt\\  -\frac{s}{2}\iint_Q\sigma_{tt}|z|^2\,dx\,dt.	
\end{equation*}	
We can see that $|\xi_t|,|\sigma_t|\leq C\xi^2$ and $|\sigma_{tt}|\leq C\xi^3$. Hence, from \eqref{eq4.3}, we have
\begin{align*}
(\!(P^-z,I_{21}+I_{23})\!) & 
\begin{multlined}[t]
\geq
Cs^3\lambda^4\iint_Q\xi^3|z|^2\,dx\,dt-Cs^3\lambda^4\int_0^T\int_0^\delta\xi^3|z|^2\,dx\,dt\\
-Cs^2\lambda^2\iint_Q\xi^3|z|^2\,dx\,dt-C(\alpha+1)s^2\lambda\iint_Q\xi^3|z|^2\,dx\,dt  -C\frac{s}{2}\iint_Q\xi^3|z|^2\,dx\,dt.	
\end{multlined}\\
& \geq Cs^3\lambda^4\left(1-\frac{1}{s_0\lambda_0^2} -\frac{1}{s_0\lambda_0^3}-\frac{1}{s_0^2\lambda_0^4}\right) \iint_Q\xi^3|z|^2\,dx\,dt-Cs^3\lambda^4\int_0^T\int_0^\delta\xi^3|z|^2\,dx\,dt
\end{align*}	

Therefore, for $\lambda_0$ and $s_0$ large enough, we have
\begin{equation}\label{eq4.4}
(\!(P^-z,I_{21}+I_{23})\!)\geq Cs^3\lambda^4\iint_Q\xi^3|z|^2\,dx\,dt-Cs^3\lambda^4\int_0^T\int_0^\delta\xi^3|z|^2\,dx\,dt.
\end{equation}

Moreover, we have that
\begin{align*}
(\!(I_{11},I_{22})\!)
&=-2s\lambda^2\iint_Q\xi x^{\alpha+2}z(x^\alpha z_x)_x\,dx\,dt\\ 
&= 2s\lambda^2\iint_Q[-\lambda\xi x^{2\alpha+3}zz_x+(\alpha+2)\xi x^{2\alpha+1}zz_x+\xi x^{2\alpha+2}|z_x|^2]\,\dd x \dd t\\ 
&\begin{multlined}[t]
= s\lambda^3\iint_Q\xi[-\lambda x^{2\alpha+4}+(2\alpha+3) x^{2\alpha+2}]|z|^2\,dx\,dt\\ 
 -(\alpha+2)s\lambda^2\iint_Q\xi[-\lambda x^{2\alpha+2}+(2\alpha+1) x^{2\alpha}]|z|^2\,dx\,dt  +2s\lambda^2\iint_Q\xi x^{2\alpha+2}|z_x|^2\,dx\,dt
\end{multlined}	 
\end{align*}
and
\begin{align*}
(\!(I_{13},I_{22})\!)
&=\iint_Qz_t(x^\alpha z_x)_x\,dx\,dt=-\iint_Qz(x^\alpha z_{tx})_x\,dx\,dt=\iint_Qx^\alpha z_xz_{xt}\,dx\,dt\\ &=\frac{1}{2}\iint_Q(x^\alpha|z_x|)_t\,dx\,dt=0.
\end{align*} Thus
\begin{equation}\label{eq4.5}
(\!(I_{11}+I_{13},I_{22})\!)\geq -Cs\lambda^4\iint_Q\xi^3|z|^2\,dx\,dt+2s\lambda^2\iint_Q\xi x^{2\alpha+2}|z_x|^2\,dx\,dt.
\end{equation}
On the other hand
\begin{align*}
	2s\lambda^2\iint_Q\xi x^{2\alpha+2}|z_x|^2\,dx\,dt&= 2s\lambda\int_0^T\left[\int_0^\delta\xi x^{2\alpha+2}|z_x|^2\,dx+\int_\delta^1\xi x^{2\alpha+2}|z_x|^2\,dx\right]\,dt\\ 
	&\geq  2s\lambda\delta^{\alpha+ 2}\int_0^T\int_\delta^T\xi x^\alpha|z_x|^2\,dx\,dt\\
	&= Cs\lambda^2\iint_Q\xi x^\alpha|z_x|^2\,dx\,dt- Cs\lambda^2\int_0^T\int_0^\delta\xi x^\alpha|z_x|^2\,dx\,dt.
\end{align*}
Hence, from \eqref{eq4.5} we deduce that
\begin{equation}\label{eq4.6}
(\!(I_{11}+I_{13},I_{22})\!)\geq Cs\lambda^2\iint_Q\xi x^\alpha|z_x|^2\,dx\,dt- Cs\lambda^2\int_0^T\int_0^\delta\xi x^\alpha|z_x|^2\,dx\,dt   -Cs\lambda^4\iint_Q\xi^3|z|^2\,dx\,dt.
\end{equation}

Finally, working as before we obtain
\begin{align*}
(\!(I_{12},I_{22})\!)
&=2s\lambda\iint_Q\xi x x^\alpha z_x(x^\alpha z_x)_x\,dx\,dt=s\lambda\iint_Q\xi x(|x^\alpha z_x|^2)_x\,dx\,dt\\ &=s\lambda^2\iint_Q\xi x^{2\alpha+2}|z_x|^2\,dx\,dt-s\lambda\iint_Q\xi x^{2\alpha}|z_x|^2\,dx\,dt +s\lambda\int_0^T\xi|z_x(1,t)|^2\,dt\\ 
&\geq Cs\lambda^2\iint_Q\xi x^\alpha|z_x|^2\,dx\,dt-Cs\lambda^2\int_0^T\int_0^\delta\xi x^\alpha|z_x|^2\,dx\,dt.
\end{align*}
Thus, from \eqref{eq4.6} we get 
\begin{equation}\label{eq4.7}
(\!(P^-z,I_{22})\!)\geq Cs\lambda^2\iint_Q\xi x^\alpha|z_x|^2\,dx\,dt- Cs\lambda^2\int_0^T\int_0^\delta\xi x^\alpha|z_x|^2\,dx\,dt -Cs\lambda^4\iint_Q\xi^3|z|^2\,dx\,dt.
\end{equation}    
  
Combining \eqref{eq4.4} and \eqref{eq4.7} we obtain that    
    \begin{equation*}
    (\!(P^-z,P^+z)\!)\geq C\iint_Q[s^3\lambda^4\xi^3|z|^2+s\lambda^2\xi x^\alpha|z_x|^2]\,dx\,dt -C\int_0^T\int_0^\delta[s^3\lambda^4\xi^3|z|^2+s\lambda^2\xi x^\alpha|z_x|^2]\,dx\,dt.
    \end{equation*}
Whence,
\begin{equation}\label{eq4.8}
   C\iint_Q[s^3\lambda^4\xi^3|z|^2+s\lambda^2\xi x^\alpha|z_x|^2]\,dx\,dt\leq  2(\!(P^-z,P^+z)\!)+C\int_0^T\int_0^\delta[s^3\lambda^4\xi^3|z|^2+s\lambda^2\xi x^\alpha|z_x|^2]\,dx\,dt.
    \end{equation}

From \eqref{eq4.2.1} and \eqref{eq4.8} we obtain    
\begin{align*}
& \|P^-z\|_2^2+\|P^+z\|_2^2+C\iint_Q[s^3\lambda^4\xi^3|z|^2+s\lambda^2\xi x^\alpha|z_x|^2]\,dx\,dt\\
& \leq 
\|P^-z\|_2^2+\|P^+z\|_2^2+2(\!(P^-z,P^+z)\!)+\int_0^T\int_0^\delta[s^3\lambda^4\xi^3|z|^2+s\lambda^2\xi x^\alpha|z_x|^2]\,dx\,dt\\
& \leq \|g\|_2^2+  \int_0^T\int_0^\delta[s^3\lambda^4\xi^3|z|^2+s\lambda^2\xi x^\alpha|z_x|^2]\,dx\,dt.
\end{align*}   
Hence, if we set $C_0=1/\min\{1,C\}$, we have that
\begin{multline*}
 \frac{1}{C_0}\left( \|P^-z\|_2^2+\|P^+z\|_2^2+\iint_Q[s^3\lambda^4\xi^3|z|^2+s\lambda^2\xi x^\alpha|z_x|^2]\,dx\,dt\right) \\ 
 \leq \|g\|_2^2+  \int_0^T\int_0^\delta[s^3\lambda^4\xi^3|z|^2+s\lambda^2\xi x^\alpha|z_x|^2]\,dx\,dt,
\end{multline*}
whence
\begin{multline}\label{eq4.9}
\|P^-z\|_2^2+\|P^+z\|_2^2+\iint_Q[s^3\lambda^4\xi^3|z|^2+s\lambda^2\xi x^\alpha|z_x|^2]\,dx\,dt\\
\leq C_0\left( \|g\|_2^2+  \int_0^T\int_0^\delta[s^3\lambda^4\xi^3|z|^2+s\lambda^2\xi x^\alpha|z_x|^2]\,dx\,dt\right)
\end{multline}   

Using \eqref{eq4.9} and the definitions of $P^-z$ and $P^+z$ one has
\begin{align}\label{eq4.10}
s^{-1}\iint_Q\xi^{-1}|z_t|^2\,dx\,dt
&\leq s^{-1}\iint_Q\xi^{-1}[|P^-z|^2+4s^2\lambda^4\xi^2x^{2\alpha+4}|z|^2+4s^2\lambda^2\xi^2x^{2\alpha+2}|z_x|^2]\,dx\,dt\nonumber \\ 
&\leq s^{-1}\|P^-z\|_2^2+Cs\lambda^4\iint_Q\xi^2|z|^2\,dx\,dt+Cs\lambda^2\iint_Q\xi x^\alpha|z_x|^2\,dx\,dt\nonumber \\
&\leq C\left( \|g\|_2^2+  \int_0^T\int_0^\delta[s^3\lambda^4\xi^3|z|^2+s\lambda^2\xi x^\alpha|z_x|^2]\,dx\,dt\right)
\end{align}
and
\begin{align}\label{eq4.11}
s^{-1}\iint_Q\xi^{-1}|(x^\alpha z_x)_x|^2\,dx\,dt
&\leq s^{-1}\iint_Q\xi^{-1}[|P^+z|^2+s^4\lambda^4\xi^4x^{2\alpha+4}|z|^2+s^2\xi^3|z|^2]\,dx\,dt\hspace{-1cm}\nonumber\\ 
&\leq s^{-1}\|P^+z\|_2^2+Cs^3\lambda^4\iint_Q\xi^3|z|^2\,dx\,dt+s\iint_Q\xi^2|z|^2\,dx\,dt\nonumber\\
&\leq C\left(\|g\|_2^2+  \int_0^T\int_0^\delta[s^3\lambda^4\xi^3|z|^2+s\lambda^2\xi x^\alpha|z_x|^2]\,dx\,dt.\right)
\end{align} 

Combining \eqref{eq4.9}-\eqref{eq4.11} we conclude that
\begin{multline}\label{eq4.12}
\iint_Q\left[s^{-1}\xi^{-1}(|z_t|^2+|(x^\alpha z_x)_x|^2)+s\lambda^2\xi x^\alpha|z_x|^2
+s^3\lambda^4\xi^3|z|^2\right]\,dx\,dt\\
\leq C\left( \|g\|_2^2+  \int_0^T\int_0^\delta[s^3\lambda^4\xi^3|z|^2+s\lambda^2\xi x^\alpha|z_x|^2]\,dx\,dt\right).
\end{multline}On the other hand, from the definition of $g$ one has \[\|g\|_2^2\leq \|e^{-s\sigma}h\|_2^2+Cs^2\lambda^4\iint_Q\xi^2|z|^2\,dx\,dt.\] Hence, for $s_0$ large enough, \eqref{eq4.12} gives
\begin{multline}\label{eq4.13}
\iint_Q\left[s^{-1}\xi^{-1}(|z_t|^2+|(x^\alpha z_x)_x|^2)+s\lambda^2\xi x^\alpha|z_x|^2
+s^3\lambda^4\xi^3|z|^2\right]\,dx\,dt\\
\leq C\left( \|e^{-s\sigma}h\|_2^2+  \int_0^T\int_0^\delta[s^3\lambda^4\xi^3|z|^2+s\lambda^2\xi x^\alpha|z_x|^2]\,dx\,dt\right).
\end{multline}

Now let us consider $\delta_1\in(\delta,d)$ and take a cut off function $\psi\in C^\infty([0,1])$ such that $0\leq\phi\leq1$, $\psi=1$ in $[0,\delta]$ and $\psi=0$ in $[\delta_1,1]$. For any $\epsilon>0$ we have that
\begin{align*}
s\lambda^2\int_0^T\int_0^\delta\xi x^\alpha|z_x|^2\,dx\,dt
&\leq s\lambda^2\int_0^T\int_0^{\delta_1}\xi\psi x^\alpha|z_x|^2\,dx\,dt\\ 
&= \int_0^T\int_0^{\delta_1}\left[s\lambda^3\xi \psi x^{\alpha+1}z_xz-s\lambda^2\xi\psi'x^\alpha z_xz-s\lambda^2\xi \psi(x^\alpha z_x)_xz\right]\,dx\,dt\\ 
&\begin{multlined}[t]
\leq C\epsilon^{-1}s^3\lambda^4\int_0^T\int_0^{\delta_1}\xi^3|z|^2\,dx\,dt+\iint_Q[s^2\lambda^4\xi^2|z|^2+\lambda^2x^\alpha |z_x|^2]\,dx\,dt\\ 
+\epsilon s^{-1}\iint_Q\xi^{-1}|(x^\alpha z_x)_x|^2\,dx\,dt.	
\end{multlined} 
\end{align*}
Hence, taking $\epsilon$ small enough and $s_0$ large enough, from \eqref{eq4.13} we conclude that
\begin{multline*}
\iint_Q\left[s^{-1}\xi^{-1}(|z_t|^2+|(x^\alpha z_x)_x|^2)+s\lambda^2\xi x^\alpha|z_x|^2
+s^3\lambda^4\xi^3|z|^2\right]\,dx\,dt\\ 
\leq C\left( \|e^{-s\sigma}h\|_2^2+  s^3\lambda^4\int_0^T\int_0^{\delta_1}\xi^3|z|^2\,dx\,dt\right).
\end{multline*} 

Using classical and well known arguments, we can coming back to the original variable $v$ and finish the proof.\hfill$\Box$

It is well known that a observability inequality for solutions of \eqref{prob0'} leads to Theorem \ref{main}. So, it is sufficient to prove the following inequality: 

\begin{proposicao}[Observability inequality]
Assume \eqref{geohip} and \eqref{hipfunc1}. There exists a constant $C>0$ such that, for any $v_T\in L^2(0,1)$ and $v$ solution of \eqref{prob0'} with $h=0$, one has
	\begin{equation}
	\label{obs}|v(\cdot,0)|_2^2\leq C\iint_{\omega_T}e^{-2s\sigma}\xi^3|v|^2\,dx\,dt.
	\end{equation}
\end{proposicao}

\begin{proof}
	From Theorem \ref{theo1} we have that 
\begin{equation}\label{eq5.0}
    s^3\lambda^4\iint_Qe^{-2s\sigma}\xi^3|v|^2\,dx\,dt\leq Cs^3\lambda^4\int_0^T\int_\omega e^{-2s\sigma}\xi^3|v|^2\,dx\,dt.
\end{equation}	
	
Multiplying the equation in \eqref{prob0'} by $v$ and integrating on $(0,1)$ we obtain that
\[-\frac{1}{2}\frac{d}{dt}|v(\cdot,t)|_2^2+\int_0^1x^\alpha|v_x|^2\,dx=\int_0^1b_1x^{\alpha/2}v_xv\,dx-\int_0^1b_0|v|^2\,dx.\]
Hence
\[-\frac{1}{2}\frac{d}{dt}|v(\cdot,t)|^2+\frac{1}{2}\int_0^1x^\alpha|v_x|^2\,dx\leq C|v(\cdot,t)|^2.\] 
Thus 
\begin{equation}\label{eq5.1}
    |v(\cdot,0)|_2^2\leq e^{2Ct}|v(\cdot,t)|^2 \ \ \ \forall t\in(0,T).
\end{equation}
	
Integrating \eqref{eq5.1} on $(T/4,3T/4)$ and using \eqref{eq5.0} we deduce that
\begin{align*}
    |v(\cdot,0)|_2^2&=\frac{2}{T}\int_{T/4}^{3T/4}|v(\cdot,0)|^2\,dt\leq C\int_{T/4}^{3T/4}\int_0^1|v|^2\,dx\,dt\\ 
    &\leq  C\int_{T/4}^{3T/4}\int_0^1s^3\lambda^4e^{-2s\sigma}\xi^3|v|^2\,dx\,dt\leq C\int_0^T\int_\omega e^{-2s\sigma}\xi^3|v|^2\,dx\,dt.
\end{align*}
\end{proof}

\section{The degenerate semilinear problem}\label{provasemilinear}
As we have pointed out in the introduction, in \cite{xu2020null} the authors proved a null controllability result for system \eqref{prob1} with $\alpha\in(0,2)$. However, most of the arguments  in that work does not depend on $\alpha$. Indeed, the only result in that paper that only works for $\alpha\in(0,2)$ is an observability estimate for system \eqref{prob0} of \cite{cannarsa2006regional}. In \eqref{obs}, we give a such estimate that works for $\alpha\geq2$. So, the majority of the arguments of \cite{xu2020null} can now be adapted to deal with \eqref{prob1} with $\alpha\geq2$. For readers convenience, we will reproduce their main guideline, but we will not present the proof of the results.

Firstly, for each $w\in L^2(0,T;H_\alpha^1)$, let us set the following notations
$$b_0[w](x,t)=\int_0^1g_s(x,t,\lambda w(x,t),\lambda w_x(x,t))\,d\lambda \ \ \ \ \mbox{and} \ \ \ \ b_1[w](x,t)=x^{-\alpha/2}\int_0^1g_p(x,t,\lambda w(x,t),\lambda w_x(x,t))\,d\lambda.$$ From \eqref{hipfunc2} we have \begin{equation}
    \label{eq4.1}\|b_0[w]\|_\infty + \|b_1[w]\|_\infty\leq 2K \ \ \ \forall w\in L^2(0,T;H_\alpha^1).
\end{equation} 
Furthermore,
\begin{equation}
    \label{eq4.2}g(x,t,u,u_x)=b_0[u](x,t)u(x,t)+x^{\alpha/2}b_1[u](x,t)u_x(x,t) \ \ \ \forall u\in L^2(0,T;H_\alpha^1) \ \ \mbox{and} \ \ a.e. \ \ \mbox{in} \ \ Q.
\end{equation}

As we will see, from \eqref{eq4.2} we can develop a fixed point argument to prove Theorem \ref{main1}.

For now, let's assume that $u_0\in H_\alpha^1$ and for each $\varepsilon>0$ consider the functional $J_\varepsilon:L^2(Q)\to\mathbb{R}$ given by \[J_\varepsilon(h)=\frac{1}{2}\int_0^T\int_{\omega}|h|^2\,dx\,dt+\frac{1}{2\varepsilon}\int_0^1|u(x,T)|^2\,dx,\]
where $u$ is the solution of \eqref{prob0} with $f=h$. The first step is to establishes a approximate null controllability result for the linear system:
\begin{proposicao}\label{prop5.1}
Assume that $u_0\in H_\alpha^1$ and \eqref{geohip}. Then, there exists $C>0$ (that does not depends on $\varepsilon$) and $h_\varepsilon\in L^2(Q)$ such that
\begin{enumerate}
    \item  $J_\varepsilon(h_\varepsilon)\leq J_\varepsilon(h) \ \ \ \forall h\in L^2(Q)$;

\item  $\int_0^T\int_\omega|h_\varepsilon|^2\,dx\,dt\leq C|u_0|^2$;

\item  if $u_\varepsilon$ is the solution of \eqref{prob0} with $f=h_\varepsilon$, then $|u_\varepsilon(\cdot,T)|\leq \varepsilon.$
\end{enumerate}


\end{proposicao}

The idea of the proof of Proposition \ref{prop5.1} is to verify that the minimum point of $J_\varepsilon$ is precisely $h_\varepsilon=-\varphi_\varepsilon1_\omega$, where $\varphi_\varepsilon$ is the solution of the adjoint system of \eqref{prob0}, with final datum $\varphi_\varepsilon(x,T)=\frac{1}{\varepsilon}u_\varepsilon(x,T)$. Then, it is possible to work with the adjoint equation to obtain the estimates given in the items 2 and 3. 

Now, a standard argument based on the Schauder's Fixed Point Theorem can be applied to obtain an approximate null controllability result for the semilinear system \eqref{prob1}.
\begin{proposicao}\label{prop5.2}
Assume that $u_0\in H_\alpha^1$ and \eqref{geohip}. Then, for each $\varepsilon>0$ there exists $h_\varepsilon\in L^2(Q)$ and $C>0$ (that does not depends on $\varepsilon$) such that:
\begin{enumerate}
    \item $\int_0^T\int_\omega|h_\varepsilon|^2\,dx\,dt\leq C|u_0|^2$;

\item  if $u_\varepsilon$ is the solution of \eqref{prob1} with $f=h_\varepsilon$, then $|u_\varepsilon(\cdot,T)|\leq \varepsilon.$
\end{enumerate}

\end{proposicao}

As we have said at the beginning of this section, the detailed proofs of  Propositions \ref{prop5.1} and \ref{prop5.2} can be found in \cite{xu2020null}. Proposition \ref{prop5.2} allows us to prove a null controllability result for the semilinear system \eqref{prob1}, with the initial data in $H_\alpha^1$.

\begin{proposicao}\label{prop5.3}
Assume that $u_0\in H_\alpha^1$ and \eqref{geohip}. Then the system \eqref{prob1} is null controllable. 
\end{proposicao}

\begin{proof}

Given $\varepsilon>0$, let us take the control $h_\varepsilon$ and the solution $u_\varepsilon$ given by Proposition \ref{prop5.2}. From item 1, there exists $\bar{h}\in L^2(Q)$ such that $h_\varepsilon\rightharpoonup\bar{h}$ in $L^2(Q)$. Furthermore, using item 1 and the energy estimates from Theorem \ref{theo1}, we can deduce that $|u_\varepsilon|^2_{\mathcal N}\leq C|u_0|^2$. Thus, there also exists $\bar{u}\in \mathcal N$ such that $u_\varepsilon\rightharpoonup\bar{u}$ in $\mathcal N$. From the compact embedding $\mathcal N\hookrightarrow \mathcal M$, we conclude that $u_\varepsilon\to\bar{u}$ in $\mathcal M$. Then $\bar{u}$ is the solution of \eqref{prob1} with $f=\bar{h}$ and, from item 2, $\bar{u}(\cdot, T)=0$.
\end{proof}

Finally, we are ready to prove Theorem \ref{main1}.

\begin{proof}[{\it Proof of Theorem \ref{main1}}] Let $u^1$ be the weak solution of the following system
\begin{equation}\label{prob1.1}
       \begin{cases}
          u_t-(x^\alpha u_x)_x+g(x,t,u,u_x)  =  0 &\text{in }    (0,1)\times(0,T_0),\\ u(1,t)=0  \text{ and }  (x^\alpha u_x)(0,t)=0  & \text{in }  (0,T_0),\\ u(x,0)= u_0(x)  & \text{in } (0,1),
       \end{cases}
\end{equation}where $T_0\in(0,T)$.

Now, let us consider the following system
\begin{equation}\label{prob1.2}
       \begin{cases}
          u_t-(x^\alpha u_x)_x+g(x,t,u,u_x)  =  h1_{\omega} &\text{in }    (0,1)\times(T_0/2,T)\\ u(1,t)=0  \text{ and }  (x^\alpha u_x)(0,t)=0  & \text{in }  (T_0/2,T),\\ u(x,T_0/2)= u^1(x,T_0/2)  & \text{in } (0,1).
       \end{cases}
\end{equation}

From Theorem \ref{theo1}, $u^1(\cdot,T_0/2)\in H_\alpha^1$. Hence, from Proposition \ref{prop5.3}, there exists a control $h^1\in L^2((0,1)\times(T_0/2,T))$ such that the associated weak solution $u^2$ of \eqref{prob1.2} satisfies $u^2(\cdot,T)=0$ in $(0,1)$. Now we can take $u\in C([0,T];L^2(Q))$ and $h\in L^2(Q)$ given by
$$u(x,t)=\left\{\begin{array}{ccc}u^1(x,t)& \mbox{if}& t\in[0,T_0/2]\\ u^2(x,t)&\mbox{if}& t\in[T_0/2,T]\end{array}\right. \ \ \ \mbox{and} \ \ \ h(x,t)=\left\{\begin{array}{ccc}0& \mbox{if}& t\in[0,T_0/2]\\ h^1(x,t)&\mbox{if}& t\in[T_0/2,T]\end{array}\right..$$

It is easy to see that $u\in \mathcal M$ is the solution of \eqref{prob1}, with $f=h$, satisfying $u(\cdot,T)=0$.
\end{proof}

\section{The degenerate nonlocal problem}\label{provanaolocal}

In this section, we will obtain the local null controllability for the problem \eqref{pbNL}. The proof is based on a meticulous inverse function argument, as specified later on.

\subsection{Functional Spaces}

The remainder of this section is devoted to a brief explanation about the most important strategies to prove Theorem \ref{th-nonlocal}.
At this point, \textit{Lyusternik's Inverse Mapping Theorem} (see \cite{alekseev1987optimal,fursikov1996}, for instance) is our main tool. Let us recall its statement:
\begin{teorema}[Lyusternik]\label{lyusternik}
	Let $E$ and $F$ be two Banach spaces, consider $H\in C^1 (E,F)$ and put $\eta_0=H(0)$. If $H'(0)\in \mathcal L (E,F)$ is onto, then there exist $r>0$ and $\tilde{H}:B_r(\eta_0)\subset F\to E$ such that
	\[		H(\tilde{H}(\xi))=\xi,\ \forall \xi\in B_r(\eta_0),\]
	which means that $\tilde{H}$ is a right inverse of $H$ in $B_{r} (\eta _0 )$. In addition, there exists $K>0$ such that
	\[\n{\tilde{H}(\xi)}{E}\leq K\n{\xi-\eta_0}{F}, \forall \xi\in B_r(\eta_0).  \]
\end{teorema}

To be more precise, let us indicate how the proof of Theorem \ref{th-nonlocal} can be seen as an application of Theorem $\ref{lyusternik}$. Even though we have not set the desired Hilbert spaces $E$ and $F$ yet, let us put
\begin{equation}\label{H}
H (u,h) = (H_{1} (u,h), H_2(u,h)),
\end{equation}
where
\[ H_{1} (u,h) := u_t-\ell \left(\int_0^1u \   \right)\left(a u_x \right)_x -f\chi_\omega\ \  \text{ and }\ \  H_2(u,h):=u(0,\cdot).\]
We should notice that, for $u_0 \in H_{\alpha}^{1}$, the first and the last relations in \eqref{pbNL} are satisfied if, and only if, there exists $(u,h)\in E$ solving
\[
H(u,h) = (0,u_0).
\]
From this point, we realize that, among other properties, $E$ and $F$ must be built:

\begin{itemize}
	\item considering the boundary conditions mentioned in \eqref{pbNL};
	\item having some imposition on its elements assuring that $u(\cdot , T)\equiv 0$. It is done having in mind some modified weights which come from \eqref{carlA}. 
 {We remark that these new weights}  {exponentially explode at} $t=T$;

\item having in mind that we want $H^{\prime} (0,0)\in \mathcal L(E,F)$ to be onto.

 In fact, we can see that
	\[H'(0,0)(u,h)=(u_t-\ell(0)(au_x)_x-f\chi_\omega,u(0)).\]
	Recalling we have assumed that $\ell(0)=1$, $H^{\prime} (0,0)\in \mathcal L(E,F)$ is onto if, and only if, given any $(g,u_0 )\in F$, the linear system 
	\begin{equation}\label{pb-lin2}
	\left\{\begin{array}{ll}
	u_t-\left(x^\alpha u_x \right)_x=f\chi_\omega + g, & (x,t)\in Q; \\
	u(1,t)=(x^\alpha u_x)(0,t)=0,& \text{ in } (0,T),\\
	u(x,0)=u_0(x),& x\in (0,1),
	\end{array}\right.
	\end{equation}
	is globally null-controllable at $T>0$, where $f\in L^2(\omega\times (0,T))$ is the control function. Hence, it seems that $E$ should contain some information involving the well-posedness (and additional regularity) of the linear system \eqref{pb-lin2}.   
\end{itemize}
From now on, we will be focused on explicitly describing the spaces $E$ and $F$, as well as, their hilbertian norms. To do so, we consider the useful notation below.

\begin{definicao}
	Let $\delta =\delta(x,t)$ and $f=f(x,t)$ be two real-valued measurable functions  defined in $Q$, where $\delta$ is non-negative. We say that  \it{$f$ belongs to $L^2(Q;\delta)$} if $\sqrt{\delta} f\in L^2(Q)$. Moreover, the natural norm of $L^2(Q;\delta )$ will be denoted by $\n{\cdot}{\delta}$, that is, 
	\[\n{f}{\delta}=\left(\intq \delta f^2\dd x \dd t\right)^{1/2}\]
for each $f\in L^{2} (Q;\delta )$.
\end{definicao}

 In order to prove the global null-controllability for the linearized system \eqref{pb-lin2}, we  first need to establish a Carleman estimate with new weight functions that do not vanish at $t=0$. Namely, consider a function $m\in C^\infty([0,T])$ satisfying
\[\left\{\begin{array}{ll}
m(t)\geq t^4(T-t)^4, & t\in (0, T/2];\\
m(t)= t^4(T-t)^4, & t\in \left[T/2,T\right];\\
m(0)>  0, &
\end{array}
\right.\]
and define	
\begin{equation}\label{functions2} 
\tau(t):=\frac{1}{m(t)},\ \zeta (x,t):=\tau(t)e^{\lambda(1+\eta(x))} \mbox{ and } 
A(x,t):=\tau (t)\left(e^{2\lambda}-e^{\lambda(1+\eta(x))}\right),
\end{equation}
where $(t,x)\in [0,T)\times [0,1]$ (see Remark $\ref{gotozero}$).

Let us note that the adjoint system associated to \eqref{pb-lin2} is
     \begin{equation} \label{adj-jlr}
\begin{cases}
-v_t-(x^\alpha v_x)_x =  h  & \text{in }  Q,\\
v(1,t)= (x^\alpha v_x)(0,t)=0  &  \text{in }   (0,T),\\ 
v(x,T)= v_T(x)   & \text{in }   (0,1),  
\end{cases}
\end{equation}
where $h\in L^2(Q)$ and $v_T\in L^2(0,1)$. Next, we state a very convenient Carleman estimate verified by any solution of \eqref{adj-jlr}. 
    
    \begin{proposicao}\label{prop3.2}
	Assuming \eqref{geohip}, there exist $C>0$, $\lambda_0 >0$ and $s_0 >0$ such that, for any $s\geq s_0$, $\lambda\geq \lambda_0$ and $v_{_T}\in L^2(\dom)$, the corresponding solution $v$ to \eqref{adj-jlr} satisfies	
	\begin{multline}
	    	\G{v} \\
	\leq C\left(\intq e^{-2sA}|h|^2\dd x \dd t  +s^3\lambda^4\intw e^{-2sA}\zeta^{6}|v|^2\dd x \dd t   \right).\label{carlA}
		\end{multline}
\end{proposicao}

The obtainment of \eqref{carlA} is a consequence of \eqref{carleman1}, by following the same steps detailed in \cite[Proposition 4]{jrl2020EECT}.

The factors multiplying $v$ in \eqref{carlA} inspire the definition of the new weight functions 
\begin{equation}\label{novospesos}
\rho_i=e^{sA}\zeta^{-i}, \emph{ where } i=0,1,2,3.
\end{equation}
As a matter of fact, $\rhoi{1}^{-2}$ and $\rhoi{3}^{-2}$ appears in the two integrals involving $v$, while $\rhoi{2}$ was chosen  in order to satisfies $\rhoi{{2}}^2=\rhoi{1}\rhoi{{3}}.$ Besides, we have $\rhoi{3}\leq C \rhoi{2}\leq C\rhoi{1}\leq C\rhoi{0}$ and, since $\rhoi{i}\geq C_T>0$ for all $i=1,2,3$, we also know that $L^2(Q;\rhoi{i}^2)\hookrightarrow L^2(Q)$. Here, for completeness, let us state the expected observability inequality which can be derived from \eqref{carlA}.

\begin{corolario}\label{observabiliy2} 
    	Assuming \eqref{geohip},
    	there exist $C>0$, $\lambda_0 >0$ and $s_0 >0$ with the following property: given $s\geq s_0$, $\lambda\geq \lambda_0$ and $v_{_T}\in L^2(\dom)$, then the corresponding solution $v$ to \eqref{adj-jlr},  with $h\equiv0$,  satisfies
	\begin{equation}\label{obs.ineq}
	|v(\cdot,0)|_2^2\leq Cs^3\lambda^{4}\intw \rhoi{3}^{-2}|v|^{2}\dd x \dd t.
	\end{equation}
\end{corolario}

\begin{observacao}\label{gotozero}
In \eqref{functions2}, we have redefined the functions given in \eqref{functions}, replacing $\theta = \theta (t)$, which satisfies $\displaystyle\lim_{t\to 0^{+}} \theta (t) = +\infty$, by $\tau = \tau (t)$ fulfilling $\displaystyle\lim_{t\to 0^{+}} \tau (t) = \tau (0) >0$. That is a crucial point in order to guarantee that \eqref{pbNL} is locally null-controllable at $T>0$, as stated in Theorem $\ref{th-nonlocal}$. Let us clarify this point: in fact, the definition of each $\rho_i$, with $i\in \{1,2,3\}$, is based on those weights which appear in \eqref{carlA}, however, it comes from \eqref{functions2}  that $\rho_1(t)\to +\infty$, as $t\to T^{-}$, and $\rho_{1} (0) >0$ (since $m(0)>0$). Because of that, $u(x,T)=0$ for any $u\in L^2(Q;\rhoi{1}^2)$. Hence, it seems reasonable to require that, if $(u,h)\in E$, then u belongs to $L^2(Q;\rhoi{1}^2)$. 	
\end{observacao}

Finally, we are ready to define $E$ and $F$. Let us consider 
\[
\mathcal{U}:= H^1(0,T;L^2(0,1))\cap L^2(0,T;H_\alpha^2)\cap C^0([0,T];H_\alpha^1)\]
and put $\mathcal{L}u:=u_t-(x^\alpha u_x)_x$ for each $u\in \mathcal  U$. Under all these previous notations, 
we set the Hilbert spaces 
\begin{equation*}
E:=\bigg\{  (u,h)\in \mathcal{U} \times L^{2} (\omega_T ;\rhoi{3}^{2} ):   
u,  (\mathcal L u-f\chi_{\omega} ) \in L^2 (Q;\rhoi{1}^2) \bigg\},
\end{equation*}
and
\[
F:= L^{2} (Q;\rhoi{1}^{2}) \times H_{\alpha}^{1} ,\]
equipped with the norms
\begin{equation*}
\|(u,h)\|_E:= \left( \| u \| _{\rhoi{1}^{2}}^{2} + \| h \| _{\rhoi{3}^{2}}^{2} +  \| \mathcal L u - f\chi _{\omega} \| _{\rhoi{1}^{2}}^{2} + \| u(0,\cdot ) \| _{H_{\alpha}^{1}}^{2} \right) ^{1/2} ,
\end{equation*}
and
\[
\| (g,v) \|_{F} :=\left( \| g \|_{\rhoi{1}^{2}}^{2} + \| v \|_{H_{\alpha}^{1}}^{2} \right) ^{1/2} ,\]
respectively. The remainder of this work is devoted to check that the mapping $H:E\longrightarrow F$ acccomplishes everything which is required in order to apply Theorem \ref{lyusternik}.


 \subsection{Global null-controllability for the linearized system}
    The goal of this section is to prove a global  null-controllability result for the linear system \eqref{pb-lin2} and establish some important additional estimates. As previously discussed, the global  null-controllability  will guarantee that $H'(0,0)$ is surjective,   which is required by \textit{Lyusternik's Theorem}, and the additional estimates will allow us to prove that $H$ is well defined and of class $C^1$. As the first step here, let us define what we mean by a solution to the problem \eqref{pb-lin2}. 
    
 
 \begin{definicao}\label{defin-transp}
     Given $u_0\in H^1_\alpha$, $f\in L^2(\domw)$ and $g\in L^2(Q)$, we say that $u\in L^2(Q)$ is a \textit{solution by transposition} of \eqref{pb-lin2} if, for each $(h,v_T)\in L^2(Q)\times L^2(0,1)$, we have 
     \[\intq uh \dd x\dd t = \into u_0v(x,0)\dd x +\intq (f1_\omega+g)v\dd x\dd t,\]
for any $v$ solution to \eqref{adj-jlr}.
 \end{definicao}

    The main result of this section is the following:	
\begin{proposicao}
\label{linearcontrol}
	Assume \eqref{geohip}. If $u_0\in H_\alpha^1$ and $g\in L^2(Q;\rhoi{1}^2)$, then there exists a control $f\in L^2(\domw; \rhoi{{3}}^2)$ to \eqref{pb-lin2}, with associated state $u\in L^2(Q;\rhoi{1}^2)$, such that
	\begin{equation*}
	    \n{u}{\rhoi{1}^2}^2+\n{f}{\rhoi{3}^2}^2\leq C\left(\n{u_0}{H^1_\alpha}^2+\n{g}{\rhoi{1}^2}^2\right).
	\end{equation*}
	In particular, it guarantees  that \eqref{pb-lin2}  is globally null-controllable.	Furthermore, we have
	\begin{equation*}
	x^{\alpha /2} u_x\in L^2(Q;\rhoi{{2}}^2),\ u_t, (x^{\alpha} u_x)_x\in L^2(Q;\rhoi{{3}}^2)
	\end{equation*}
	and there exists $C>0$ such that
	\begin{equation}\label{est-ad}
	\n{x^{\alpha /2} u_x}{\rhoi{{2}}^2}^2+\n{u_t}{\rhoi{{3}}^2}^2+\n{(x^{\alpha} u_x)_x}{\rhoi{{3}}^2}^2\leq C\left(\n{u}{\rhoi{{1}}^2}^2+\n{h\chi_{\omega }}{\rhoi{{3}}^2}^2+\n{g}{\rhoi{1}^2}^2+\n{u_0}{H^1_{\alpha}}^2\right).
	\end{equation}
\end{proposicao}

\begin{proof}  Let us define the set
\[P_{0\alpha}=\{ w\in C^2(\bar{Q});\ w(1,t)=x^\alpha w_x(0,t)=0,\ t\in (0,T)\}.\]
Recalling the definition of $\mathcal{L}$, we can see that its formal adjoint is given by $\mathcal{L}^\ast v=-v_t-(x^\alpha v_x)_x$. Hence, analyzing the right-hand side of $\eqref{carlA}$, we can define the following symmetric, positive defined bilinear form
\begin{equation*}
    a(w_1,w_2)=\intq \rhoi{0}^{-2}\mathcal{L}^\ast w_1\mathcal{L}^\ast w_2\dd x\dd t+ \intq \rhoi{3}^{-2}w_1w_21_\omega \dd x \dd t, \ \forall w_1,w_2\in P_{0\alpha}.
    \end{equation*}
Thus, let us denote by $P_\alpha$ the completion of $P_{0\alpha}$ with respect to the inner product defined by $a$. Hence,  $P_\alpha$ is a Hilbert space with norm given by $\n{v}{P_\alpha}=a(v,v)^{1/2}$.

Now, let us define the continuous linear functional $L: L^2(Q)\longrightarrow \R$ given by \[Lv=\into u_0v(x,0)\dd x+\intq gv \dd x\dd t. \]

In this case, Lax-Milgram Theorem yields $\hat{v}\in P_\alpha$ such that
\[a(\hat{v},v)=Lv,\ \forall v\in P_\alpha,\]
that is,
\begin{equation*}
    \intq \rhoi{0}^{-2}\mathcal{L}^\ast \hat{v}\mathcal{L}^\ast v_2\dd x\dd t+ \intq \rhoi{3}^{-2}\hat{v}v1_\omega \dd x \dd t=\into u_0v(x,0)\dd x+\intq gv \dd x\dd t, \ \forall v\in P_\alpha.
\end{equation*}

According to Definition \ref{defin-transp}, it means that $f:=-\rhoi{3}^{-2}\hat{v}1_\omega$ is a control and $u:=\rhoi{0}^{-2}\mathcal{L}^\ast \hat{v}$ the associated state to the problem \eqref{pb-lin2}. Indeed, for any $(h,v_T)\in L^2(Q)\times L^2(0,1)$, if $v$ is a solution to \eqref{adj-jlr}, then
\begin{equation*}
    \intq uh\dd x\dd t= \into u_0v(x,0)\dd x+\intq (f1_\omega +g )v\dd x \dd t.
\end{equation*}

Furthermore, from Carleman and observability inequalitie, given in \eqref{carlA} and \eqref{obs.ineq} respectively, we have
\begin{align*}
    \n{\hat{v}}{P_\alpha}^2& =L\hat{v}\leq \nn{u_0} \nn{\hat{v}(\cdot,0)}+\n{g}{\rhoi{1}^2}\left(\intq \rhoi{1}^{-2}\hat{v}^2\dd x\dd t\right)^{1/2}\\
    & \leq \left(\nn{u_0}^2+\n{g}{\rhoi{1}^2}^2\right)^{1/2}\left(\nn{\hat{v}(\cdot,0)}^2+ \intq \rhoi{1}^{-2}\hat{v}^2\dd x \dd t\right)^{1/2}\\
     & \leq C\left(\nn{u_0}^2+\n{g}{\rhoi{1}^2}^2\right)^{1/2}a(\hat{v},\hat{v})^{1/2}\\
     & = C\left(\nn{u_0}^2+\n{g}{\rhoi{1}^2}^2\right)^{1/2}\n{\hat{v}}{P_\alpha},
\end{align*}
whence
\begin{equation*}
    \n{\hat{v}}{P_\alpha} \leq C\left(\nn{u_0}^2+\n{g}{\rhoi{1}^2}^2\right)^{1/2}.
\end{equation*}
Using the explicit expressions $f=-\rhoi{3}^{-2}\hat{v}1_\omega$ and $u=\rhoi{0}^{-2}\mathcal{L}^\ast \hat{v}$, as well as, recalling the norm $\n{\cdot}{P_\alpha}$, we easily get
\begin{align*}
    \n{u}{\rhoi{1}^2}^2+\n{f}{\rhoi{3}^2}^2&   \leq  C\intq \rhoi{0}^2u^2\dd x \dd t+\intq \rhoi{3}^2f^2 \dd x\dd t \\
    & =\intq \rhoi{0}^{-2}|\mathcal{L}^\ast \hat{v}|^2 \dd x \dd t +\intq \rhoi{3}^{-2}\hat{v}^21_\omega \dd x \dd t\\
    & \leq C\left(\nn{u_0}^2+\n{g}{\rhoi{1}^2}^2\right),
\end{align*}
as desired. 

At this moment, we would like to say that the obtainment of \eqref{est-ad} will be left for the two subsequent lemmas.
\end{proof}

%

\begin{lema}\label{ad1}
Assume \eqref{geohip}. Given $u_0 \in H_{\alpha}^{1}$ and $g\in L^2 (Q;\rhoi{1}^2)$, if $(u,h)\in \mathcal U \times L^2 (Q_{\omega}; \rhoi{3}^2)$ is a solution to \eqref{pb-lin2}, then $x^{\alpha /2} u_x\in L^2(Q;\rhoi{{2}})$ and there exists $C>0$ such that
	\begin{equation*}
	\n{x^{\alpha /2}u_x}{\rhoi{{2}}^2}^2\leq C\left(\n{u}{\rhoi{{1}}^2}^2+\n{h\chi_{\omega }}{\rhoi{{3}}^2}^2+\n{g}{\rhoi{1}^2}^2+\n{u_0}{H^1_{\alpha}}^2\right).
	\end{equation*}
\end{lema}

\begin{proof}
	Multiplying the equation in (\ref{pb-lin2}) by $\rhoi{2}^2u$, integrating in $[0,1]$ and using the two relations		
	\[\frac{1}{2}\frac{d}{dt}\into \rhoi{2}^2 u^2=\into\rhoi{2}^2 u_tu+\into \rhoi{2}(\rhoi{2})_tu^2\]
	and
	\begin{align*}
	\into\rhoi{2}^2(x^{\alpha /2} u_x)_xu\    & =-2\into \rhoi{2}(\rhoi{2})_x x^{\alpha} uu_x- \into\rhoi{2}^2 x^{\alpha} u_x^2,
	\end{align*}
	we obtain
	\begin{align}\label{myeq2}
	\frac{1}{2}\frac{d}{dt}\into\rhoi{2}^2u^2 \   +\into\rhoi{2}^2 x^{\alpha} u_x^2
	& = \begin{multlined}[t]
	- \into\rhoi{2}^2cu^2+\into\rhoi{2}^2uh\chi_\omega +\into\rhoi{2}^2gu \\+\into\rhoi{2}(\rhoi{2})_tu^2-2\into\rhoi{2}(\rhoi{2})_x x^{\alpha} uu_x\nonumber\\
	\end{multlined}\\ 
	& = I_1+I_2+I_3+I_4+I_5.
	\end{align}

	Now, using $\rhoi{i}\leq C\rhoi{j}$, for $i\geq j$, and $\rhoi{1}\rhoi{3}=\rhoi{2}^2$, we obtain 
	\begin{align*}
	& I_1\leq C \into\rhoi{1}^2|u|^2  ,\\
	& I_2\leq C\left(\frac{1}{2}\int_0^1\rhoi{3}^2|h\chi_\omega|^2\   +\frac{1}{2}\int_0^1\rhoi{1}^2|u|^2\   \right)
	\end{align*}
	and
	\[
	\displaystyle I_3\leq C\left(\frac{1}{2}\int_0^1\rhoi{1}^2|g|^2\   +\frac{1}{2}\int_0^1\rhoi{1}^2|u|^2\   \right).
	\]
	
	Let us estimate $I_4$. Firstly, we will rewrite $A$ as
	$A(t,x)=\varsigma(t,x) \bar{\mu}(x)$, where  
	\[\bar{\mu}(x):=(e^{M\lambda} - e^{\lambda (1+\eta (x))})/\mu(x)
	.\]
	Secondly, note that
	\[
	\rhoi{2}(\rhoi{2})_t
	=e^{sA} \varsigma ^{-2} (se^{sA} \varsigma _t \bar{\mu} \varsigma ^{-2} - 2e^{sA} \varsigma ^{-3} \varsigma _t) 
	=e^{sA} \varsigma ^{-2} (s\varsigma ^{-2} \bar{\mu} - 2\varsigma ^{-3}) \varsigma _t 
	\]
	
	Then, for all $t\in [0,T]$,
	\begin{equation*}
	|\rhoi{2}(\rhoi{2})_t| \leq C \rhoi{1}^2 \varsigma ^{-2} |\varsigma _t| \leq C \rhoi{1}^2,
	\end{equation*}
	whence
	$$I_4\leq C\into\rhoi{1}^2|u|^2  .$$
	
	Now, using
	\begin{equation*}
	|(\rhoi{2})_x|^2 x^{\alpha} u^2\leq Ce^{-2sA}\varsigma^{-2}\left|\varsigma^{-2}+\varsigma^{-4}\right| |\varsigma_x^2|x^{\alpha} u^2\leq C\rhoi{1} ^2u^2,
	\end{equation*}
	we obtain
	\begin{equation*}
	I_5 \leq 2\into|\rhoi{2} x^{\alpha /2} u_x||(\rhoi{2})_x x^{\alpha /2} u| 
	\leq \frac{1}{2}\into\rhoi{2}^2 x^{\alpha} u_x^2+2\into |(\rhoi{2})_x|^2 x^{\alpha} u^2  
	\leq \frac{1}{2}\into\rhoi{2}^2 x^{\alpha} u_x^2+C\into \rhoi{1}^2u^2.
	\end{equation*}
	Hence, (\ref{myeq2}) gives us
	\[\frac{d}{dt}\into\rhoi{2}^2|u|^2  +\into\rhoi{2}^2 x^{\alpha}|u_x|^2  
	\leq C\left(\into\rhoi{1} ^2|u|^2  +\into\rhoi{3}^2|h\chi_\omega|^2   
	+\into\rhoi{1}^2|g|^2  \right).\]
	Integrating in time, the desired result follows.
\end{proof}

\begin{lema} \label{ad2}
	Assume \eqref{geohip}.  Given $u_0 \in H_{\alpha}^{1}$ and $g\in L^2 (Q;\rhoi{1}^2)$, if $(u,h)\in \mathcal U \times L^2 (Q_{\omega}; \rhoi{3}^2)$ is a solution to \eqref{pb-lin2}, then $u_t, (au_x)_x\in L^2(Q;\rhoi{{3}}^2)$ and there exists $C>0$ such that
	\begin{equation*}
	\n{u_t}{\rhoi{{3}}^2}^2+\n{(x^{\alpha} u_x)_x}{\rhoi{{3}}^2}^2\leq C\left(\n{u}{\rhoi{{1}}^2}^2+\n{h\chi_{\omega }}{\rhoi{{3}}^2}^2+\n{g}{\rhoi{1}^2}^2+\n{u_0}{H^1_{\alpha}}^2\right).
	\end{equation*}
\end{lema}

\begin{proof}
	
	In the first step,  we will estimate the first term of left side of the inequality. Multiplying equation in  (\ref{pb-lin2}) by $\rhoi{3}^2u_t$ and integrating in $[0,1]$, we have
	
	\begin{align}\label{myeq3}
	\int_0^1\rhoi{3}^2u_t^2 
	&=  \int_0^1\rhoi{3}^2u_th\chi_\omega   +\into\rhoi{3}^2g u_t -\into c(x,t)\rhoi{3}^2uu_t  +\into\rhoi{3}^2(x^{\alpha} u_x)_xu_t\
	\nonumber \\
	& =: I_1+I_2-I_3+I_4.
	\end{align}
	
	Using Young's inequality with $\varepsilon$ and $\rhoi{i}\leq C\rhoi{j}$, for $i\geq j$, we obtain
	\begin{align*}
	I_1  &\leq \into \rhoi{3}^2|h\chi_\omega||u_t|\   
	\leq \varepsilon\into\rhoi{3}^2|u_t|^2  +\frac{1}{4\varepsilon}\into\rhoi{3}^2|h\chi_\omega|^2,\\
	I_2 
	&\leq \into\rhoi{3}^2|gu_t| 
	\leq \varepsilon\into\rhoi{3}^2|u_t|^2  +\frac{1}{4\varepsilon}\into\rhoi{3}^2|g|^2  
	\leq \varepsilon\into\rhoi{3}^2|u_t|^2  
	+C\into\rhoi{1} ^2|g|^2
	\end{align*}
	and
	\begin{align*}
	-I_3\leq \into |c(t,x)|\rhoi{3}^2|uu_t|   
	&\leq \varepsilon\into\rhoi{3}^2 u_t^2  +C\into\rhoi{1}^2u^2.
	\end{align*}
	
	Now, integrating $I_4$ by parts, we can see that
	\begin{align}\label{myeq37}
	I_4
	& =\left.\rhoi{3}^2 x^{\alpha}u_xu_t\right\vert_{x=0}^{x=1} -\into(\rhoi{3}^2u_t)_x x^{\alpha} u_x =-2\into\rhoi{3}(\rhoi{3})_x x^{\alpha} u_tu_x-\frac{1}{2}\frac{d}{dt}\into\rhoi{3}^2 x^{\alpha} u_x^2 +\frac{1}{2}\into(\rhoi{3}^2)_t x^{\alpha} u_x^2\nonumber\\
	& =-2I_{41}-\frac{1}{2}\frac{d}{dt}\into\rhoi{3}^2 x^{\alpha} u_x^2\   +\frac{1}{2} I_{42}.
	\end{align}
	Hence,
	\begin{equation}\label{myeq4}
	\int_0^1\rhoi{3}^2|u_t|^2\   +\frac{1}{2}\frac{d}{dt}\into\rhoi{3}^2 x^{\alpha} |u_x|^2\    =I_1+I_2-I_3-2I_{41}+\frac{1}{2}I_{42}.
	\end{equation}
	Since $|(\rhoi{3})_x|\leq C\rhoi{2}$ and $|(\rhoi{3}^{2})_t | \leq C\rhoi{2}^2$, observe that
	\[
	|\rhoi{3}(\rhoi{3})_x x^{\alpha} u_xu_t|\leq C|\rhoi{3} u_t||\rhoi{2} x^{\alpha/2} u_x|
	\]
	and
	$$|(\rhoi{3}^2)_t|= 2|\rhoi{3}(\rhoi{3})_t|\leq C\rhoi{2}^2.$$
	So that,
	$$I_{41}\leq \frac{1}{4}\into\rhoi{3}^2u_t^2+C\into\rhoi{2}^2 x^{\alpha} u_x^2$$
	and
	\begin{equation*}\label{myeq8}
	I_{42} \leq C\into\rhoi{2}^2 x^{\alpha} u_x^2  .
	\end{equation*}
	Using the estimates obtained for $I_1,I_2,I_3, I_{41}$ and $I_{42}$, the relation (\ref{myeq4}) provides
	
	\begin{equation*}
	\int_0^1\rhoi{3}^2u_t^2\   +\frac{1}{2}\frac{d}{dt}\into\rhoi{3}^2 x^{\alpha} u_x^2\      \leq C\left(\into\rhoi{3}^2|h\chi_\omega|^2  +\into\rhoi{1}^2g^2 \right.
	\left.+\into \rhoi{1}^2u^2  +\into\rhoi{2}^2 x^{\alpha} u_x^2  \right),
	\end{equation*}
	and, consequently,
	\begin{equation}\label{myeq9}
	\intq\rhoi{3}^2u_t^2   \leq C\left(\intq\rhoi{1}^2u^2  +\intw\rhoi{3}^2h^2  \right. 	\left.+\intq\rhoi{1}^2g^2
	+\|u_0\|_{H_{\alpha}^1}^2\right).
	\end{equation}
	
	In the second part, we must estimate the term $\intq\rhoi{3}^2|(x^{\alpha} u_x)_x|^2  $. Multiplying the equation in (\ref{pb-lin2}) by $-\rhoi{3}^2(x^{\alpha} u_x)_x$ and integrating in $[0,1]$, we take
	\begin{align*} 
	\into\rhoi{3}^2|(x^{\alpha} u_x)_x|^2   
	&=-\into\rhoi{3}^2h\chi_\omega(x^{\alpha} u_x)_x  
	-\into\rhoi{3}^2g(x^{\alpha} u_x)_x  
	+\into c(x,t)\rhoi{3}^2u(x^{\alpha} u_x)_x  +\into\rhoi{3}^2u_t(x^{\alpha}u_x)_x  \\
	& =-J_1-J_2+J_3+I_4.
	\end{align*}
	
	As earlier in this proof, applying Young's inequality with $\varepsilon$, we obtain
	\[J_1
	\leq \into\rhoi{3}^2|h\chi_\omega||(x^{\alpha} u_x)_x|  \leq \varepsilon\into\rhoi{3}^2|(x^{\alpha} u_x)_x|^2  +\frac{1}{4\varepsilon}\into\rhoi{3}^2|h\chi_\omega|^2  .\]
	\[J_2
	\leq \into\rhoi{3}^2|g||(x^{\alpha} u_x)_x|  
	\leq \varepsilon\into\rhoi{3}^2|(x^{\alpha} u_x)_x|^2  +\frac{1}{4\varepsilon}\into\rhoi{1}^2g^2 .\]
	\begin{align*}
	J_3 & \leq C\left(\varepsilon\into\rhoi{3}^2|(x^{\alpha}u_x)_x|^2  +\frac{1}{4\varepsilon}\into\rhoi{1}^2u^2  .\right)
	\end{align*}
	
	From \eqref{myeq37} and \eqref{myeq9}, we achieve
	\begin{equation*}
	\into\rhoi{3}^2|(x^{\alpha} u_x)_x|^2  +\frac{1}{2}\frac{d}{dt}\into\rhoi{3}^2 x^{\alpha} |u_x|^2  \leq C\left(\into\rhoi{3}^2|h\chi_\omega|^2  +\into\rhoi{1}^2|g|^2  \right.
	\left.+\into\rhoi{1}^2|u|^2  +\into\rhoi{2}^2 x^{\alpha}|u_x|^2  \right)
	\end{equation*}
	Integrating in time, we conclude the proof.
\end{proof}

\subsection{Local null-controllability for the nonlinear system}
	In this section, our goal is to prove Theorem \ref{th-nonlocal}. It pass by applying Theorem \ref{lyusternik}, which will allow us to conclude that $H:E\longrightarrow F$, given in \eqref{H}, has a right inverse mapping defined in a small ball $B\subset F= L^2(Q;\rho_1^2)\times H_a^1$. Since Theorem \ref{linearcontrol}   already guarantees that $H'(0,0)\in \mathcal L(E,F)$ is onto, it remains to verify that
		\begin{itemize}
			\item $H$ is well-defined;
			\item $H\in C^1 (E,F)$.
		\end{itemize}
	
		We will clarify that in Propositions \ref{lema3.3} amd \ref{lema3.4}.
\begin{proposicao}\label{lema3.3}
	The mapping $H:E\longrightarrow F$, given in \eqref{H}, is well defined.
\end{proposicao}
\begin{proof}
	Given $(u,h)\in E$, we intend to prove that $H(u,h)$ belongs to $L^2(Q;\rhoi{1}^2)\times H_\alpha^1$. From the definition of $E$, it is clear that $H_2(u,h)=u(0,\cdot)\in H^1_\alpha$. Let us see that $H_1 (u,h)\in L^2(Q;\rhoi{1}^2)$. 
	
	In fact, since $\ell(0)=1$ and $\ell$ is Lipschitz continuous, we have
	\begin{align*}
	& \intq\rz^2|H_1(u,h)|^2 \dd x \dd t = \intq \rz^2\left|u_t-\ell\left(\int_0^1u\dd x \right)(x^\alpha u_x)_x-h\chi_\omega\right|^2   \\
	& \leq 4\intq \rz^2\left| \mathcal L (u)-h\chi_\omega\right|^2\dd x \dd t   +4 \intq \rz^2\left|\left[\ell\left(\int_0^1u\dd x\right)-\ell(0)\right]
	(x^\alpha u_x)_x\right|^2 \dd x \dd t \\
	& \leq 4\n{(u,h)}{E}^2+4\intq \rhoi{1}^2\left(\into u \dd x\right)^2 |(x^\alpha u_x)_x|^2 \dd x \dd t.
	\end{align*}
	
	Hence, we just need to prove that the last integral is bounded from above by $\n{(u,h)}{E}^2$. Indeed, note that
	\begin{align*}
	    \intq \rhoi{1}^2\left(\into u \dd x\right)^2 |(x^\alpha u_x)_x|^2 \dd x \dd t & =\intq \rhoi{1}^2\rhoi{3}^{-2}\left(\into u \dd x\right)^2 \rhoi{3}^2|(x^\alpha u_x)_x|^2 \dd x \dd t\\
	    & \leq C  \sup_{[0,T]}\left(\tau^4\left(\into u \dd x\right)^2\right)\intq  \rhoi{3}^2|(x^\alpha u_x)_x|^2 \dd x \dd t\\
	    & \leq C \sup_{[0,T]}\left(\tau^4\left(\into u \dd x\right)^2\right)\n{(u,h)}{E}^2 \\
	    &\leq C\n{(u,h)}{E}^4,
	\end{align*}
where the last inequality is a consequence of Lemma \ref{supremo}, since $\tau^4\leq Ce^{M_s/m(t)}$.  \end{proof}

	\begin{lema}\label{supremo}
	Given $s>0$, there exists $M_s >0$ such that 
	\[
	\displaystyle \sup_{t\in[0,T]} \left\lbrace e^{\frac{M_s}{m(t)}} \left( \int_{0}^{1} u \right)^2  \right\rbrace
	\leq C\|(u,h)\|_{_E}^{2},\]
	for all $(u,h)\in E$, where $m=m(t)$ is the the function defined in \eqref{functions2}.
\end{lema}

\begin{proof}
	Firstly, for $s>0$, let us consider $(u,h)\in E$ and the function $q:[0,T]\longrightarrow \mathbb{R}$
\[
\displaystyle q(t):=e^{\frac{M_s}{m(t)}} \left( \int_{0}^{1} u(x,t)dx\right)^2,
\]
for all $t\in [0,T]$, where $M_s >0$ will be specified later. \\

\textbf{\underline{Claim 1}:} Given $s>0$, there exist $M_s >0$ and $C>0$ such that 
\[
\displaystyle e^{\frac{M_s}{m(t)}} \leq C \rhoi{1} ^2.
\]
\\
Indeed, for any $K>0$, we know that
\[
\displaystyle e^{\frac{-k}{m}} \leq \frac{2}{k^2} [m(t)]^2 \emph{ for all } t\in [0,T].
\]
In particular, taking $k=s\beta _{\lambda}$ and $M_s = \frac{s\beta _{\lambda}}{2}$, we obtain 
\begin{equation}
    \rhoi{1}^2 = e^{2sA} \varsigma ^{-2} \geq e^{-2\lambda} m^2 e^{2sA} \geq \frac{e^{-2\lambda} k^2}{2} e^{2sA-\frac{k}{m}} = C_{\lambda , s} e^{\frac{2s\beta_{\lambda} -k}{m}} = C_{\lambda ,s} e^{\frac{2M_s}{m}},
\end{equation}
onde $C_{\lambda ,s} = \frac{e^{-2 \lambda} s^2 \beta _{\lambda }^2}{2}$. \\

\textbf{\underline{Claim 2}:} There exist $K_1 = K_1 (\lambda ,s)>0$ and $K_2 = K_2 (\lambda ,s)>0$, such that 
\begin{equation}
\frac{\rhoi{3}^2}{m^2} \leq K_1 \rhoi{1}^2 \emph{ and } e^{\frac{2M_s}{m}} \leq K_2 \rhoi{3}^2.
\end{equation}
As a consequence, $q\in H^1(0,T)\hookrightarrow C^0([0,T])$. \\

In fact, arguing as in Claim 1, we can get
\begin{equation*}
    \displaystyle \frac{\rhoi{3}^2}{m^4} = \frac{e^{2sA} \tau^{-2}}{\mu ^2} \leq \frac{\rhoi{1}^2}{\mu ^6} \leq K_1 \rhoi{1}^2
\end{equation*}
and
\begin{equation*}
    \displaystyle \rhoi{3}^2 = \frac{e^{2sA} m^6}{\mu^{6}} \geq e^{-6\lambda} e^{2sA} \frac{k^6}{6!} e^{\frac{-k}{m}} 
    = \frac{e^{-6\lambda} k^6}{6!} e^{\frac{2s\beta_{\lambda} -k}{m}} 
    =\frac{1}{K_2} e^{\frac{2M_s}{m}},
\end{equation*}
where we have taken $k=s\beta_{\lambda}$, $M_s = \frac{s\beta_{\lambda}}{2}$ and $K_2 = \frac{6!}{e^{-6\lambda} (s\beta _{\lambda})^6}$. In this case, 
\begin{equation*}
    \displaystyle \int_{0}^{T} |q|^2  
    \leq \int_{0}^{T} \int_{0}^{1} e^{\frac{2M_s}{m}} |u|^2 
    \leq \frac{1}{C_{\lambda ,s}} \int_{0}^{T} \int_{0}^{1} \rhoi{1}^2 |u|^2 
    \leq C\|(u,h)\|_{E}^{2}
\end{equation*}
and
\begin{align*}
\displaystyle \int_{0}^{T} |q^{\prime}|^2
&\leq C\left( \int_{0}^{T} \int_{0}^{1} \frac{M_{s}^2 (m')^2}{m^4} e^{\frac{2M_s}{m}} |u|^2 + \int_{0}^{T} \int_{0}^{1} e^{\frac{2M_s}{m}} |u_t|^2 \right) \\
&\leq C\left( \int_{0}^{T} \int_{0}^{1} \frac{\rhoi{3}^2}{m^4} |u|^2
+ \int_{0}^{T} \int_{0}^{1} \rhoi{3}^2 |u_t|^2 \right) \\
&\leq C\left( \int_{0}^{T} \int_{0}^{1} \rhoi{1}^2 |u|^2
+ \int_{0}^{T} \int_{0}^{1} \rhoi{3}^2 |u_t|^2 \right) \\
&\leq C\|(u,h)\|_{E}^{2},
\end{align*}
following the desired result.
\end{proof}

\begin{proposicao}
    \label{lema3.4}
	The mapping $H$ belongs to $C^1 (E,F)$.
\end{proposicao}

\begin{proof}
    It is clear that $H_2\in C^1$. We will prove that $H_1$ has a continuous Gateaux derivative on $E$. In fact, some well-known calculation allows us to see that the Gateaux derivative of $H_1$ at $(u,h)\in E$ is given by 
	\[H_{1}^{\prime} (u,h) (\bu,\bh):=\bu_t-\ell'\left(\into u\dd x\right)\into\bu \dd x\, (x^\alpha u_x)_x-\ell\left(\into u\dd x \right)(x^\alpha \bu_x)_x-\bh\chi_\omega,
	\]
	for each $(\bu,\bh)\in E$. We just need to prove that the Gateaux derivative $H'_1:E\to \mathcal{L}(E;L^2(Q;\rhoi{1}^2) )$ is continuous. On this purpose, given  $(u,h)\in E$, let $((u^{n} ,h^{n}) )_{n=1}^{\infty}$ be a sequence in $E$ such that $\| (u^{n} ,h^{n} )-(u,h) \| _{E} \rightarrow 0$. We must prove that $\| H'_1 (u^{n} ,h^{n} )-H'_1 (u,h) \| _{\mathcal L (E;L^2(Q;\rhoi{1}^2))} \rightarrow 0$. In fact, taking $(\bu ,\bh) $ on the unit sphere of $E$, we can see that
	\begin{align*}
	    & \| (H'_1 (u^{n} ,h^{n} ) - H^{\prime}_1 (u,h) )(\bu ,\bh)\| _{\rho_{1}^{2}}^{2}\\
	    &\begin{multlined}[t]
	    =\intq \rhoi{1}^2\bigg| -\ell'\left(\into u^n\dd x\right)\into\bu \dd x\, (x^\alpha u^n_x)_x-\ell\left(\into u^n\dd x \right)(x^\alpha \bu_x)_x   \dd x \dd t\\+\ell'\left(\into u\dd x\right)\into\bu \dd x\, (x^\alpha u_x)_x+\ell\left(\into u\dd x \right)(x^\alpha \bu_x)_x  \dd x \dd t\bigg|^2
	    \end{multlined}\\
	    & \begin{multlined}[t]
	    \leq C\intq \rhoi{1}^2\left(\into\bu \dd x\right)^2\left(\ell'\left(\into u^n\dd x\right)\right)^2 |(x^\alpha (u^n_x-u_x))_x |^2\dd x \dd t\\
	    +C\intq \rhoi{1}^2\left(\into\bu \dd x\right)^2\left(\ell'\left(\into u^n\dd x\right)-\ell'\left(\into u \dd x\right)\right)^2 |(x^\alpha u^n_x)_x |^2\dd x \dd t\\
	    +C\intq \rhoi{1}^2\left(\into\bu \dd x\right)^2\left(\ell\left(\into u^n\dd x\right)-\ell\left(\into u \dd x\right)\right)^2 |(x^\alpha \bu_x)_x |^2\dd x \dd t.
	    \end{multlined}
	\end{align*}
Proceeding as in \cite{jrl2020EECT}, using that $\ell \in C^1 (\R,\R)$ has bounded derivatives and applying Lebesgue's dominated convergence theorem, we can prove that each of these three last integral converges to zero, as $n\to +\infty$. Hence, $H_{1}^{\prime}$ is continuous, as desired.\end{proof}

\begin{proof}[Proof of Theorem $\ref{main1}$] 
		
		We already know that the mapping $H:E \longrightarrow F$ is well defined and belongs to $C^1(E,F)$ (Propositions $\ref{lema3.3}$ and $\ref{lema3.4}$). We state that $H^{\prime} (0,0)\in \mathcal L (E,F)$ is onto. In fact,  given $(g,u_0) \in F=L^2(Q;\rhoi{1}^2)\times H_\alpha^1$, we apply Proposition $\ref{linearcontrol}$ in order to obtain $(u,h)\in L^2(Q;\rhoi{1}^2)\times L^2(\domw;\rhoi{3}^2)$ which solves \eqref{pb-lin2} and satisfies \eqref{est-ad}. It means that $(u,h)\in E$ and $H^{\prime}(0,0)(u,h)=(g,u_0)$, as desired.
		
		Hence, by \textit{Lyusternik's Inverse Mapping Theorem} (Theorem \ref{lyusternik}) , there exist $\varepsilon>0$ and a mapping $\tilde{H} :B_{\varepsilon} (0) \subset L^2(Q;\rhoi{1}^2)\times H_\alpha^1 \longrightarrow E$ such that
		\[
		\displaystyle H(\tilde{H} (y)) = y \emph{  for each  } y\in B_{\varepsilon}(0)\subset L^2(Q;\rhoi{1}^2)\times H_\alpha^1 .
		\]
		In particular, if $\bar{u}_0\in H_{\alpha}^{1}$ and $\| \bar{u}_0\|_{H_{\alpha}^{1}} < \varepsilon$, we conclude that $(\bar{u},\bar{h})=\tilde{H} (0,\bar{u}_0)\in E$ solves  $H(\bar{u},\bar{h}) = (0,\bar{u}_0)$. Finally, since $\bar{u}\in L^2(Q;\rhoi{1}^2)$, we get $\bar{u}(x,T)=0$ almost everywhere in $[0,1]$ (see Remark $\ref{gotozero}$). It completes the proof.
	\end{proof}

\bibliography{references}

\end{document}